\documentclass[12pt,a4paper]{article}

\usepackage{amsmath,amstext,amssymb,amscd,euscript}
 \usepackage{graphicx}

\usepackage[cp1251]{inputenc}
\usepackage[russian,english]{babel}

\newcommand{\Xcomment}[1]{}
\newtheorem{theorem}{Theorem}[section]
\newtheorem{lemma}[theorem]{Lemma}
\newtheorem{corollary}[theorem]{Corollary}
\newtheorem{prop}[theorem]{Proposition}

\newenvironment{proof}{\noindent{\bf Proof}\/}%
{\hfill$\qed$\medskip}

\def\qed{\ \vrule width.2cm height.3cm depth0cm}

\makeatletter \@addtoreset{equation}{section} \makeatother

\newenvironment{numitem1}{\refstepcounter{equation}\begin{enumerate}%
\item[(\thesection.\arabic{equation})]}{\end{enumerate}}

\newcommand{\refeq}[1]{(\ref{eq:#1})}  

 \makeatletter
\renewcommand{\section}{\@startsection{section}{1}{0pt}%
{-3.5ex plus -1ex minus -.2ex}{2.3ex plus .2ex}%
{\normalfont\Large}}
 \makeatother

 \makeatletter
\renewcommand{\subsection}{\@startsection{subsection}{2}{0pt}%
{-3.0ex plus -1ex minus -.2ex}{1.5ex plus .2ex}%
{\normalfont\normalsize\bf}}
 \makeatother

 \Xcomment{
 \makeatletter
\renewcommand{\subsection}{\@startsection{subsection}{2}{0pt}%
{-3.0ex plus -1ex minus -.2ex}{-1.5ex plus .2ex}%
{\normalfont\normalsize\bf}}
 \makeatother
 }

 \newcommand{\SEC}[1]{\ref{sec:#1}}  
\newcommand{\SSEC}[1]{\ref{ssec:#1}}  

\def\Rset{{\mathbb R}}

\def\Zset{{\mathbb Z}}

\def\Ascr{{\cal A}}

\def\Escr{{\cal E}}

\def\Iscr{{\cal I}}

\def\Lscr{{\cal L}}

\def\Pscr{{\cal P}}

\def\Rscr{{\cal R}}
\def\Sscr{{\cal S}}
\def\Tscr{{\cal T}}

\def\Xscr{{\cal X}}

\def\tilde{\widetilde}
\def\hat{\widehat}

\def\deltain{\delta^{\rm in}}
\def\deltaout{\delta^{\rm out}}

\def\Xmin{X^{\rm min}}
\def\Xmax{X^{\rm max}}

\def\Imax{I^{\rm max}}

\def\convex{{\rm conv}}

\def\rest#1{_{\,\vrule height 1.5ex width 0.05em depth 0pt\, #1}}

\begin{document}
\baselineskip=14pt
\parskip=3pt

\title{Stable matchings, choice functions, and linear orders}

\author{
Alexander V. Karzanov\thanks{Central Institute of Economics and Mathematics of
the Russian Acad. Sci., 47, Nakhimovskii Prospect, 117418 Moscow, Russia; email: akarzanov7@gmail.com}
}

\date{}

 \maketitle

\vspace{-0.7cm}
\begin{abstract}
We consider a model of stable edge sets (``matchings'') in a bipartite graph $G=(V,E)$ in which the preferences for vertices of one side  (``firms'') are given via choice functions subject to standard axioms of consistency, substitutability and cardinal monotonicity, whereas the preferences for the vertices of the other side  (``workers'') via linear orders. For such a model, we present a combinatorial description  of the structure of rotations and develop an algorithm to construct the poset of rotations, in time $O(|E|^2)$ (including oracle calls). As consequences, one can obtain a ``compact'' affine representation of stable matchings and efficiently solve some related problems.
 
 \end{abstract}


\section{Introduction}  \label{sec:intr}

Studies in the field of stable contracts in bipartite markets have been started with the prominent work by Gale and Shapley~\cite{GS} on stable marriages. In their model and its natural extensions (of types ``one-to-many'' or ``many-to-many''), one considers a bipartite graph $G=(V,E)$ in this the vertices are interpreted as ``agents of a market'', and the edges as ``possible contracts'' involving pairs of agents. The preferences of an agent $v\in V$ within the set of its accessible contracts (viz. incident edges) $E_v$ are given via a linear order, the number of constracts allowed to be chosen by $v$ is restricted by a given \emph{quota} $q(v)$, and a set  $X\subseteq E$ of contracts obeying the quotas is regarded as \emph{stable} if no contract from the complement $E-X$ is more preferred for both sides compared with some of those chosen by them.

Since then, stability problems in bipartite graphs with linear preferences on the vertices have deserved popularity in the literature; for extensive surveys on results in this direction, see e.g.~\cite{GI,manl}. Among important properties established for such problems are the following: a stable contract exists for arbitrary quotas; the set of stable contracts constitutes a distributive lattice (subject to natural comparisons on them); the optimal stable contract for each of both sides of the market can be constructed efficiently. (Hereinafter, we say that an algorithm is \emph{efficient} if, roughly, the number of standard operations, or \emph{time}, that it spends is bounded by a polynomial in $|V|,|E|$.)

For convenience, in the (fixed) partition of the vertex set $V$ into two parts (viz. independent sets, color classes), we denote these parts as $W$ and $F$, and call the elements in these as \emph{workers} and \emph{firms}, respectively. (In the classical model with unit quotas considered in~\cite{GS}, the vertices of different parts are interpreted as ``men'' and ``women''.) We liberally refer to an \emph{arbitrary} subset of edges $X\subseteq E$ as a \emph{matching}, and to a stable subset of edges as a \emph{stable matching}. Emphasize that in this work we will deal with matchings only in \emph{bipartite} graphs.

A considerably larger class of models of stability in bipartite graphs arises by replacing linear preferences by ones determined by choice functions. Here for each vertex $v\in V$ a \emph{choice function} (briefly, CF) is an operator $C_v$ on $2^{E_v}$ that, given a subset $Z\subseteq E_v$, chooses in it an ``acceptable'' (more preferred) part $C_v(Z)\subseteq Z$. As a rule, such a choice function $C_v$ obeys the axioms of \emph{consistency} and \emph{substitutability}, which leads to building a reasonable theory of stable matchings generalizing basic properties  for models with linear preferences. (Note also that imposing these two axioms are equivalent to fulfilling the property of  \emph{path independence} going back to Plott~\cite{plott}.) Initial steps in development of this theory appeared in the 1980s; we can especially distinguish the works of Kelso and Crawford~\cite{KC}, Roth~\cite{roth}, and Blair~\cite{blair}. In particular, it was shown there that the set of stable matchings is nonempty and forms a lattice. (Note also  that the above pair of axioms is equivalent to the property of \emph{path independence} going back to Plott~\cite{plott}.)

An important contribution to this theory was due to Alkan in the early 2000s, who showed in~\cite{alk1,alk2} that adding, to the above two axioms, one more axiom of \emph{quota filling}, or, weaker, \emph{cardinal monotonicity}, turns the lattice of stable matchings into a distributive one. This, in light of Birkhoff's representation theorem~\cite{birk}, gives rise to a possibility to represent the stable matchings via ideals of a certain poset. 

Earlier a representation of this sort was described in the simplest case -- for the stable marriages -- by Irving and Leither. They showed in~\cite{IL} that the corresponding poset is formed by the so-called \emph{rotations}, cycles determining transformations of stable marriages into ``neighboring'' ones, that the number of rotations does not exceed $|E|$, and that the poset of rotations (giving a ``compact representation'' of the set of stable marriages) can be constructed efficiently. (At the same time, it was shown in~\cite{IL} that the task of computing the number of stable marriages in a bipartite graph is intractible, $\#P$-hard.) In the subsequent work~\cite{ILG}, the authors explain that using the poset of rotations, one can efficiently solve the problem of minimizing a linear function over the set of stable marriages; this attracts a method of Picard~\cite{pic}, giving a reduction to the classical minimum cut problem in an edge-weighted graph.

Recently Faenza and Zhang~\cite{FZ} conducted an in-depth study of rotations, their posets and applications for Alkan's general models in~\cite{alk1,alk2}. They considered stable matchings generated by choice functions in two situations: 1)  CFs  for all vertices are plottian (viz. consistent and substitutable) and cardinally monotone, and 2) in addition, CFs for the vertices of one side, say, $W$, are quota-filling (while those for the other side, $F$, are cardinally monotone, in general). We will use the name \emph{general boolean model} (GBM) in the first case, and \emph{special boolean model} (SBM) in the second case. Here the term ``boolean'' reflects the fact that we deal with stable subsets in $E$, or, equivalently, with stable $0,1$ functions on $E$, in contrast to models where stable functions with wider values (say, real or arbitrary integer ones) are admitted.

As the main results in~\cite{FZ} concerning GBM, one can distinguish the following: a refinement of the structure of rotations is described (which need not be formed by simple cycles, in general); one shows that the set $\Rscr$ of rotations consists of $O(|F| |W|)$ elements; a bijection $X\stackrel{\omega}{\longmapsto} I$ between the lattice $(\Sscr,\succ)$ of stable matchings $X$ and the lattice $(\Iscr,\subset)$ of ideals $I$ of the poset of rotations is established; one shows that the map $\omega^{-1}$ gives an integer affine representation of the stable matchings via the ideals of the above poset. The last means the existence of an $E\times\Rscr$ matrix  $A$ with integer entries providing the relation $x=Au+x^0$, where $x$ and $u$ are the characteristic vectors (in the spaces $\Rset^E$ and $\Rset^\Rscr$, respectively) of a stable matching and an ideal related via $\omega$, and $x^0$ is the characteristic vector of the $W$-optimal stable matching. Moreover, one shows that the matrix $A$ has full column rank, which implies that the polytope of stable matchings is affinely congruent to an order polytope as in Stanley~\cite{stan}.

These results are strengthened for SBM. More precisely, in~\cite{FZ} one shows that in this case the poset of rotations and the matrix $A$ of affine representation can be constructed efficiently. Here one assumes that the choice functions are given via \emph{oracles}, so that when asking the oracle (viz. CF) $C_v$ about an arbitrary set $Z\subseteq E_v$, it outputs the value $C_v(Z)$ in ``oracle time'' polynomial in $|Z|$. One may liberally assume that this oracle time is measured as a constant $O(1)$ (similar to a wide scope of problems where one is interested in the number of ``oracle calls'' rather than the complexity of their implementations). Under these assumptions,~\cite{FZ} estimates the running time of computing the above poset and matrix as $O(|F|^3 |W|^3)$. Like the models with linear preferences, this implies that one can apply to SBM an efficient method to solve a linear program on the set of stable matchings.

In this paper we consider a special case of SBM. Here, like the full version of SBM, the preferences in the sides of $G$ are defined differently. More precisely, each vertex $v$ in the side $F$ is equipped with a choice function $C_v$ on $2^{E_v}$ which is plottian and cardinally monotone (similar to GBM and SBM). As before, all these CFs are given by oracles. In turn, each vertex $v$ in the side $W$ is equipped with a quota $q(v)$, and the preverences on $E_v$ are defined by a linear order. We conditionally refer to the model of stability in this case as the \emph{combined boolean} one (briefly, CBM). 

It should be noted that the model of this sort arises as a result of a reduction from a stability problem in a bipartite graph in which the preferences for one side, $F$, are given via general CFs (as above) whereas those for the other side, $W$, via so-called \emph{sequential choice functions}, as described in~\cite{dan}.

The main aim of our study of CBM in this paper is to elaborate relatively simple and enlightening methods of constructing rotations and their poset,  proving the bijection between the lattices of stable matchings and ideals of the poset, and more. Our constructions give rise to efficient algorithms having relatively small time bounds (to compare: we construct the poset of rotations in time $O(|E|^2)$, whereas a similar task for SBM in~\cite{FZ} is performed in $O(|F|^3|W|^3)\approx O(|E|^3)$ time). 

The paper is organized as follows.

Section~\SEC{def} contains basic definitions and settings concerning CBM. Section~\SEC{active} reviews assertions and tools needed to introduce rotations in our case. In Section~\SEC{poset} we explain the construction of poset of rotations and prove that the lattices of stable matchings and ideals of this poset are isomorphic. Section~\SEC{build} mainly concerns algorithmic aspects. Here we show that the poset of rotations can be constructed in time $O(|E|^2)$, and the initial stable matching (optimal for $W$) in time $O(|V| |E|)$. Relying on the construction of the rotational poset, in Section~\SEC{affine} we demonstrate affine representation for the set of stable matchings, and give a facet description of the polytope of stable matchings. It should be noted that in order to simplify the description, throughout Sections~\SEC{active}--\SEC{affine}, we prefer to deal with the case of \emph{unit quotas} for all vertices in the part $W$. In Section~\SEC{general}, this is generalized to arbitrary quotas under linear preferences in $W$ (which does not cause big efforts). In the concluding Section~\SEC{sequent}, following~\cite{dan}, we recall the notion of sequential choice functions and a reduction of the stability model with such functions to CBM, and then we briefly describe consequences from our results obtained for the latter model to the former one.


\section{Basic definitions and settings}  \label{sec:def}

In the model of our study, we are given: a bipartite graph $G=(V,E)$ in which the vertex set $V$ is partitioned into two parts (independent sets) $F$ and $W$, called the sets of \emph{firms} and \emph{workers}, respectively. Without loss of generality, one may assume that $G$ has no multiple edges (see Remark~1 in the end of this section); also one may assume that $G$ is connected. In particular, $|V|-1\le |E|\le \binom{|V|}{2}$. The edge of $G$ connecting vertices $w\in W$ and $f\in F$ can be denoted as $wf$. 

For a vertex $v\in V$, the set of incident edges is denoted by $E_v$. On this set, preferences of the vertex (``agent'') $v$ are defined. Preferences for the parts $F$ and $W$ are arranged in a different way.
 \smallskip

$\bullet$ (\textbf{linear preferences}) For a vertex $w\in W$, the preferences are given by use of a \emph{linear order} $>_w$ on $E_w$. If for $e,e'\in E_w$, there holds $e>_w e'$, then we say that the edge $e$ is preferred to $e'$ for $w$.  This is similar to the definition of preferences in the classical problem of stable marriages due to Gale and Shapley~\cite{GS}.
 \smallskip

$\bullet$ (\textbf{choice functions}) For $f\in F$, the preferences use a \emph{choice function} (CF) $C=C_f:2^{E_f}\to 2^{E_f}$. It satisfies several standard conditions (axioms). First of all, $C$ is assumed to be a non-expensive operator, in the sense that for any $Z\subseteq E_f$, there holds $C(Z)\subseteq Z$. Two axioms concerning pairs $Z,Z'\subseteq E_f$ are viewed as follows:
  \begin{itemize}
\item[(A1)] if $Z\supseteq Z'\supseteq C(Z)$, then $C(Z')=C(Z)$;
\item[(A2)] if $Z\supseteq Z'$, then $C(Z)\cap Z'\subseteq C(Z')$.
  \end{itemize}

In particular, (A1) implies that any $Z\subseteq E_f$ satisfies
$C(C(Z))=C(Z)$. In literature, property~(A1) is usually called the
\emph{consistency}, and property~(A2) the \emph{substitutability}, or
\emph{persistence} (the last term is encountered, e.g., in~\cite{AG}). As is shown in~\cite{AM}, validity of both (A1) and (A2) is equivalent to fulfilling the property of \emph{path independence}, or \emph{plottianness} (the term justified by~\cite{plott}); in our case, this is viewed as: 
  \begin{numitem1} \label{eq:plott}for any $Z,Z'\subseteq E_f$, there holds $C(Z\cup Z')=C(C(Z)\cup Z')$.
   \end{numitem1}

One more axiom is known under the name of \emph{cardinal monotonicity}, saying that:
 \begin{itemize}
 \item[(A3)] if $Z\supseteq Z'$, then $|C(Z)|\ge |C(Z')|$.
 \end{itemize}

An important special case of~(A3) is the condition of \emph{quota-filling}; it is applied when there is a prescribed number (\emph{quota}) $q(f)\in \Zset_{>0}$, and is viewed as:
 \begin{itemize}
 \item[(A4)]  any $Z\subseteq E_f$ satisfies $|C(Z)|=\min\{|Z|,q(f)\}$.
 \end{itemize}

One easily checks that the above axioms are valid for the vertices $w\in W$; in this case, each vertex $w$ is endowed with a quota $q(w)\in\Zset_{>0}$, and the operator $C_w$ acts according to the order $>_w$, namely: in a set $Z\subseteq E_w$, it takes $\min\{q(w),|Z|\}$ largest elements w.r.t. $>_w$.

In our further description, up to Section~\SEC{general}, we will consider, as a rule, the special case of CBM in which the quotas $q(w)$ of all vertices $w\in W$ are ones (which makes our description simpler and more enlightening); at the same time, the choice functions $C_f$, $f\in F$, are assumed to be arbitrary subject to conditions (A1),(A2),(A3).
\medskip

$\bullet$ (\textbf{stability})
It will be convenient to us to name any subset $X\subseteq E$ of edges of $G$ as a \emph{matching}. For a vertex $v\in V$, the restriction of $X\subseteq E$ to $E_v$ is denoted as $X_v$; in other words, $X_v=X\cap E_v$. A subset $Z\subseteq E_v$ is called  \emph{acceptable} if $C_v(Z)=Z$; the collection of such subsets is denoted by $\Ascr_v$. This notion is extended to subsets of the whole $E$; namely, we say that $X\subseteq E$ is acceptable if so are all of its restrictions $X_v$, $v\in V$. The collection of all acceptable subsets (matchings) of $E$ is denoted by $\Ascr$.

For any $v\in V$, the choice function $C_v$ enables us to compare acceptable subsets of $E_v$. Namely, for $Z,Z'\in \Ascr_v$, let us say that $Z$ is preferred to $Z'$, denoting this as $Z\succ_v Z'$, if 
   $$
   C_v(Z\cup Z')=Z.
   $$
One can see that the relation $\succ_w$ is transitive.

Based on the comparison of acceptable subsets in the sets $E_v$, one can compare acceptable matchings in the whole $E$. Namely, taking one of the two parts in $G$, say, $F$, for distinct matchings $X,Y\in \Ascr$, we write $X\succ_F Y$ and say that $X$ is preferred to $Y$ relative to the ``firms'' if $X_f\succeq_f Y_f$ holds for all $f\in F$. The order in $\Ascr$ w.r.t. the ``workers'' is defined in a similar way and denoted as $\succ_W$. 
 \medskip

\noindent\textbf{Definition.} For $v\in V$, we say that an edge $e\in E_v$ is \emph{interesting} relative to (or under) an acceptable set $Z\in\Ascr_v$ if $e\in E_v-Z$ and $e\in C_v(Z\cup \{e\})$. This notion is extended to acceptable matchings in $E$. More precisely, for a matching $X\in\Ascr$, an edge $e=wf\in E-X$ is called \emph{interesting} under $X$ for a vertex (``agent'') $v\in\{w,f\}$ if $e\in C_v(X_v\cup \{e\})$. If an edge $e=wf\in E-X$ is interesting under $X$ for both vertices $w$ and $f$, we say that $e$ \emph{blocks} $X$. A matching $X\in \Ascr$ is called \emph{stable} if no edge in $E-X$ blocks $X$. The set of stable matchings for $G=(F\sqcup W,E)$, $>_w$, $q(w)$ ($w\in W$), $C_f$ ($f\in F$) in question is denoted by $\Sscr=\Sscr(G,>,q,C)$.
\medskip

Note that
  \begin{numitem1} \label{eq:Ze}
for $v\in V$, $Z\in\Ascr_v$, $e\in E_v-Z$, and $Z':=C_v(Z\cup\{e\})$:
  \begin{itemize}
\item[(i)] $e$ is interesting under $Z$ if and only if $Z'$ is equal to either (a) $Z\cup\{e\}$, or (b) $(Z-\{e'\})\cup\{e\}$ for some $e'\in Z$;
\item[(ii)] if $e$ is interesting under $Z$, then $Z'\succ_v Z$, and vice versa.
 \end{itemize}
 \end{numitem1}

\noindent Here (i) follows from the relations $C_v(Z)=Z$ and $Z'\subseteq
Z\cup\{e\}$ and the inequality $|C_v(Z\cup\{e\})| \ge |C_v(Z)|$ (in view of the cardinal monotonicity (A3)). Property~(ii) follows from $C_v(Z'\cup Z)=C_v(Z\cup\{e\})=Z'\ne Z$.
  \medskip

\noindent$\bullet$ ~For $v\in V$, the set $\Ascr_v$ endowed with the preference relation $\succ_v$ turns into a lattice; according to facts in~\cite{alk2}, in this lattice, for $Z,Z'\in\Ascr_v$, the least upper bound (``join'') $Z\curlyvee Z'$ is expressed as $C_v(Z\cup Z')$ (while the ``meet'' $Z\curlywedge Z'$ is expressed by using the notion of closure of a set; this is not needed for our purposes and omitted here).

The ``direct sum'' of the lattices $(\Ascr_f,\succ_f)$ over $f\in F$ gives lattice $(\Ascr,\succ_F)$ (in which for $X,X'\in\Ascr$, the join $X\curlyvee_F X'$ and meet $X\curlywedge_F X'$ w.r.t. the part $F$ are defined in a natural way via the restrictions $X_f\curlyvee X'_f$ and $X_f\curlywedge X'_f$ for $f\in F$). The lattice $(\Ascr,\succ_W)$ w.r.t. the part $W$ is defined in a similar way.
 \medskip

\noindent$\bullet$ ~Now we formulate three basic properties of the set $\Sscr$ of stable matchings for CBM, which are important to us; they follow directly from corresponding general results in Alkan's work~\cite{alk2} (see Theorem~10 there). They are as follows:
  \begin{numitem1} \label{eq:stab-mix}
 \begin{itemize}
\item[(a)] the set $\Sscr$ is nonempty, and $(\Sscr,\succ_F)$ forms a 
\emph{distributive} lattice;
\item[(b)] (property of \emph{polarity}): the partial order $\succ_F$ is opposite to the order $\succ_W$, namely: for $X,Y\in\Sscr$, if  $X\succ_F Y$ then
$Y\succ_W X$, and conversely;
\item[(c)] (property of \emph{unicardinality}): for each fixed vertex $v\in V$, the size $|X_v|$ is the same for all stable matchings $X\in \Sscr$.
 \end{itemize}
 \end{numitem1}
We will denote the minimal and maximal elements in the lattice
$(\Sscr,\succ_F)$ by $\Xmin$ and $\Xmax$, respectively (so the former is the best and the latter is the worst for the part $W$, in view of polarity~\refeq{stab-mix}(b)).

Note also that in case of quota-filling (which, in particular, takes place for the part $W$), property~(c) in~\refeq{stab-mix} can be strengthened as follows (cf.~\cite[Corollary~3]{AG}):
  \begin{numitem1} \label{eq:thesame}
for a vertex $v\in V$, if the choice function $C_v$ obeys axiom (A4), and if some (equivalently, any) stable matching $X$ satisfies $|X_v|<q(v)$, then the restriction $X_v$ is the same for all $X\in \Sscr$.
  \end{numitem1}
  
\noindent\textbf{Remark 1.} Sometimes in the literature on stable matchings in graphs (not necessarily bipartite), one assumes the presence of multiple edges, i.e. an input graph $G$ is allowed to be a ``multigraph'' (admitting multiple edges may have reasonable ``economic'' interpretations). Nevertheless, there is a simple transformation of $G$ into a graph $G'$ free of multiple edges, giving an equivalent stable matching problem with $G'$. Such a construction is pointed out in~\cite{CF} for the case of linear preferences in all vertices; this is applicable to our model as well. (So we may always restrict ourselves by consideration of usual graphs, which simplifies a description, leading to no loss of generality; in particular, an edge can be encoded by the pair of its endvertices.) This elegant construction deserves to be briefly outlined.  According to it, an arbitrary (or desirable) edge $e$ connecting vertices $u$ and $v$ in the graph $G$ is replaced by subgraph $K_e$ generated by 6-cycle $O_e$ formed by the sequence of (new) vertices $v_1,\ldots,v_6$ and two additional edges $uv_1$ and $vv_4$. The preferences in the vertices $v_i$ are defined in a circular manner: the edge $v_{i-1}v_i$ is preferred to $v_iv_{i+1}$ (letting $v_6=v_0$); in addition, the edge $uv_1$ ($vv_4$) is assigned to have the middle preference in the triple for $v_1$ (respectively, $v_4$). The quotas of all new vertices are $v_i$ are assigned to be 1. 

One easily checkes that in a stable matching $X'$ for the model with $G'$, all new vertices $v_i$ must be covered by $X'$. It follows that $X'$ contains either none or both edges among $uv_1$ and $vv_4$ as above. This induces, in a natural way, a matching $X$ for $G$, which is easily shown to be stable for the model with $G$. Conversely, a stable matching $X$ for $G$ can be transformed into a stable matching $X'$ for $G'$; here both edges $uv_1$ and $vv_4$ are included in $X'$ if and only if the edge $e$ belongs to $X$. One can see that within the subgraph $K_e$, the matching $X'$ is defined uniquely, except for the case when $e\notin X$ and, in addition, $e$ is interesting under $X$ for none of the vertex $u$ and $v$ (we distinguish this as \emph{case} ($\ast$)). Then $X'$ can be assigned within $K_e$ in two ways: putting either $v_1v_2, v_3v_4, v_5v_6\in X'$, or $v_2v_3, v_4v_5, v_6v_1\in X'$. An interrelation between the rotations for $G'$ and $G$ will be outlined in Remark~2 in Section~\SEC{build}.


\section{Active graph and rotations}  \label{sec:active}

Throughout this section, we fix a matching $X\in\Sscr$ different from $\Xmax$. We are interested in the set $\Sscr_X$ of stable matchings $X'$ that are close to $X$ and satisfy $X'\succ_F X$. This means that $X$ \emph{immediately precedes} $X'$ in the lattice $(\Sscr,\succ_F)$; in other words, there is no $Y\in\Sscr$ between $X$ and $X'$, i.e. such that $X'\succ_F Y\succ_F X$. In order to find $\Sscr_X$ we construct the so-called \emph{active graph} and then extract in it special cycles called \emph{rotations}. Our method of constructing the rotations is essentially simpler and faster than an analogous method in~\cite{FZ} elaborated for a more general model.
 \medskip

\noindent\textbf{Definition 1.} A vertex $w\in W$ is called \emph{deficit} if $|X_w|<q(w)$. Otherwise (when $|X_w|=q(w)$)  $w$ is called \emph{fully filled}; the set of such vertices is denoted as $W^=$.
  \medskip

\noindent(In light of property~\refeq{stab-mix}(c), the set $W^=$ does not depend on $X\in\Sscr$. Also, by~\refeq{thesame}, for a deficit vertex $w$, the set $X_w=X\cap E_w$ does not depend on $X\in\Sscr$.)

In what follows, when no ambiguity can arise, we will write $\succ$ for $\succ_F$. To simplify our description, unless otherwise is explicitly said, we will consider the case of unit quotas $q(w)=1$ for all $w\in W$; general quotas $q$ on $W$ will be postponed to Section~\SEC{general}.


\subsection{Active graph.} \label{ssec:active_g}
For a fully filled vertex $w\in W^=$, let $x_w$ denote the unique edge of the matching $X$ incident to $w$, i.e. $X_w=\{x_w\}$. Consider the set of edges $e=wf\in E_w$ satisfying the following two properties:
  \begin{numitem1} \label{eq:W-admiss}
(a) $x_w>_w e$, and (b) $e$ is interesting for $f$ under $X$, i.e. $e\in
C_f(X_f\cup \{e\})$.
  \end{numitem1}

When this set is nonempty, the best (most preferred) edge in it is called $W$-\emph{admissible} for $X$ and denoted as $a_w=a_w(X)$.

Consider such an edge $a_w=wf$. According to~\refeq{Ze}(i), two cases for $X_f$ and $a_w$ are possible. If $C_f(X_f\cup\{a_w\})$ is expressed as
$(X_f-\{e'\})\cup\{a_w\}$ for some $e'\in X_f$ (case~(b)), then we say that the edge
$e'$ is $F$-\emph{admissible} and denote it as $b_f^w=b_f^w(X)$. Also we say that $b_f^w$ is \emph{associated} with $a_w$, and that the pair $(a_w,b_f^w)$ forms a \emph{tandem} passing the vertex $f$.

(When $C_f(X_f\cup\{a_w\})=X_f\cup\{a_w\}$ (cf. case~(a) in~\refeq{Ze}(i)), the edge $a_w$ generates no tandem. Note also that some $F$-admissible edges may be associated with two or more $W$-admissible edges, i.e. different tandems $(a_w,b_f^w)$ and $(a_{w'},b_f^{w'})$ with $b_f^w=b_f^{w'}$ are possible.) It follows immediately from the definitions of admissible edges that
  \begin{numitem1} \label{eq:ab}
any tandem $(a,b)$ passing a vertex $f\in F$ satisfies
$a\notin X_f$, $b\in X_f$, and $C_f(X_f\cup\{a\})=(X_f\cup\{a\})-\{b\}$.
  \end{numitem1}

Let $D=D(X)=(V,E_D)$ be the directed graph whose edge set consists of
the $W$-admissible edges directed from $W$ to $F$ and  the $F$-admissible edges directed from $F$ to $W$; we use notation of the form $(w,f)$ for the former, and $(f,w)$ for the latter edges (where $w\in W$ and $f\in F$). Apply to $D$ the following procedure.

  \begin{description}
\item[\emph{Cleaning procedure}:] When scanning a vertex $w\in W$ in a current graph $D$, if we observe that $w$ has leaving $W$-admissible edge $a_w=(w,f)$ but no entering $F$-admissible edge (of the form $(f',w)$), then we delete $a_w$ from $D$. Simultaneously, if such an $a_w$ has the associated edge $b_f^w$ and this edge does not occur in other tandems passing $f$ in the current $D$, then we delete $b_f^w$ from $D$ as well. Along the way, we remove from $D$ each isolated (zero degree) vertex whenever it appears. Repeat the procedure with the updated $D$, and so on, untill  $D$ stabilizes.
  \end{description}

Let $\Gamma=\Gamma(X)=(V_\Gamma,E_\Gamma)$ denote the graph $D$ upon termination of the above procedure. We refer to $\Gamma$ as the \emph{active graph} for $X$, and its edges as the \emph{active edges}.  Define $W_\Gamma:=W\cap V_\Gamma$ and $F_\Gamma:=F\cap V_\Gamma$. For a vertex $v\in V_\Gamma$, denote the set of edges of $\Gamma$ leaving $v$ (entering $v$) by $\deltaout(v)=\deltaout_\Gamma(v)$ (resp. by $\deltain(v)=\deltain_\Gamma(v)$. The graph $\Gamma$ possesses the following nice properties (of balancedness):

 \begin{numitem1} \label{eq:balance}
(a) each vertex $w\in W_\Gamma$ satisfies $|\deltaout(w)|=|\deltain(w)|=1$; (b) each vertex $f\in F_\Gamma$ satisfies $|\deltaout(f)|=|\deltain(f)|$, and the tandems $(a_w,b_f^w)$ passing $f$ are edge disjoint and form a partition of the set
$\deltaout(f)\cup\deltain(f)$ (where $a_w\in\deltain(f)$ and $b_f^w\in\deltaout(f)$).
  \end{numitem1}

Indeed, from the definitions of $W$- and $F$-admissible edges we observe that
$|\deltain(w)|,|\deltaout(w)|\le 1$ for all $w\in W_\Gamma$, and
$|\deltain(f)|\ge|\deltaout(f)|$ for all $f\in F_\Gamma$. As a result of the Cleaning procedure, the inequalities in the former expression turn into
$1=|\deltain(w)|\ge|\deltaout(w)|$, while the form of the latter expression (as being an inequality where the l.h.s. is greater than or equal to the r.h.s.) preserves. Now the desired properties follow from evident balance relations (since in $\Gamma$, the edges leaving $W$ and the ones entering $F$ are the same, and similarly for the edges leaving $F$ and the ones entering $W$).


\subsection{Rotations.} \label{ssec:rotation}
From~\refeq{balance} it follows that the active graph $\Gamma=\Gamma(X)$ is decomposed into a set of pairwise edge disjoint directed cycles, where each cycle $L=(v_0,e_1,v_1,\ldots,e_k,v_k=v_0)$ is uniquely constructed in a natural way, namely: for $i=1,\ldots,k$, if $v_i\in W_\Gamma$, then $\{e_i\}=\deltain(v_i)$ and $\{e_{i+1}\}=\deltaout(v_i)$, and if $v_i\in F_\Gamma$, then the pair $(e_i,e_{i+1})$ forms a tandem passing $v_i$. Note that all edges in $L$ are different, but $L$ can be self-intersecting in vertices of the part $F$ (so $L$ is \emph{edge simple} but not necessarily simple). Depending on the context,  the cycle $L$ may also be regarded as a subgraph in $\Gamma$ and denoted as $L=(V_L,E_L)$.

Let $\Lscr=\Lscr(X)$ be the set of above-mentioned cycles in $\Gamma$. For each $L\in\Lscr$, define $(L^+,L^-)$ to be the partition of the set of its edges, where $L^+$ consists of the edges going from $W$ to $F$, and $L^-$ of the ones going from $F$ to $W$ (called, respectively, $W$-\emph{active} and $F$-\emph{active} edges of $L$). The cycles $L\in\Lscr$ are just what we call the \emph{rotations} generated by the matching $X$. The key properties of rotations are exhibited in the following two assertions.
  \begin{prop} \label{pr:XXp}
For each $L\in\Lscr(X)$, the matching $X':=(X-L^-)\cup L^+$ is stable and satisfies $X'\succ X$.
  \end{prop}

We say that the matching $X'$ defined in this proposition \emph{is obtained from $X$ by applying the rotation} $L$, and denote the set of these over all $L\in\Lscr(X)$ by $\Sscr_X$.
  \begin{prop} \label{pr:XYXp}
Let $Y\in \Sscr$ and $X\prec Y$. Then there exists $X'\in\Sscr_X$ such that $X'\preceq Y$.
  \end{prop}

The proofs of these propositions will essentially use the next lemma. Hereinafter, to simplify notation, for a subset of edges $Z$ and elements $a\notin Z$ and $b\in Z$, we may write $Z+a$ for $Z\cup \{a\}$, and $Z-b$ for $Z-\{b\}$.

 \begin{lemma} \label{lm:a1-ak}
Let $f\in F$. Let $(a(1),b(1)), \ldots,(a(k),b(k))$ be different (pairwise edge disjoint) tandems in $\Gamma$ passing $f$. Then for any $I\subseteq \{1,\ldots,k\}=:[k]$, there holds
  $$
  C_f(X_f+a(1)+\cdots+a(k)-\{b(i)\colon i\in I\})
   =X_f+a(1)+\cdots+a(k)-b(1)-\cdots-b(k).
  $$
  \end{lemma}
  \begin{proof}
~Denote $X_f+a(1)+\cdots+a(k)-\{b(i)\colon i\in I\}$ by $Z_I$. We have to show that $C_f(Z_I)=Z_{[k]}$ for any $I\subseteq[k]$.

First we show this for $I=\emptyset$. To this aim, we compare the actions of $C_f$ on $Z_\emptyset=X_f+a(1)+\cdots+a(k)$ and on $Y_i:=X_f+a(i)$ for an arbitrary 
$i\in[k]$. The definition of the tandem $(a(i),b(i))$ implies $C_f(Y_i)=X_f+a(i)-b(i)$. Applying axiom~(A2) to the pair $Z_\emptyset\supset Y_i$, we have
  $$
  C_f(Z_\emptyset)\cap Y_i\subseteq C_f(Y_i)=X_f+a(i)-b(i).
  $$
Since $Y_i$ contains $b_i$, it follows that $b_i\notin C_f(Z_\emptyset)$.

Thus, $C_f(Z_\emptyset)\subseteq Z_\emptyset-b(1)-\cdots-b(k)=Z_{[k]}$. Here the inclusion should turn into equality, which follows from the cardinal monotonicity applied to the pair $Z_\emptyset\supseteq X_f$, giving
  $$
  |C_f(Z_\emptyset)|\ge|C_f(X_f)|=|X_f|=|Z_{[k]}|.
  $$

Now consider an arbitrary $I\subseteq[k]$. Then $Z_\emptyset\supseteq Z_I\supseteq C_f(Z_\emptyset)=Z_{[k]}$, and applying axiom~(A1), we obtain $C_f(Z_I)=Z_{[k]}$, as required.
  \end{proof}

This lemma implies that
  \begin{numitem1} \label{eq:Xfab}
the set $Z:=X_f+a(1)+\cdots+a(k)-b(1)-\cdots-b(k)$ is acceptable for $f$ and satisfies $Z\succ_f X_f$.
  \end{numitem1}
Since Lemma~\ref{lm:a1-ak} applied to $I=\emptyset$ gives $C_f(Z\cup X_f)=C_f(X_f+a(1)+\cdots+a(k))=Z$.
 \medskip

\noindent\textbf{Proof of Proposition~\ref{pr:XXp}.} 
First of all, note that the set $X'_v$ is acceptable for all $v\in V$. This follows from the acceptability of $X_v$ if $v\notin V_L$. For $w\in W\cap V_L$, the set $X'_w$ consists of a unique edge, and $C_w(X'_w)=X'_w$ is obvious. And for $f\in F\cap V_L$, the acceptability of $X'_f$ follows from~\refeq{Xfab} when $L^+\cap
E_f=\{a_1,\ldots,a_k\}$ and $L^-\cap E_f=\{b(1,\ldots,b(k)\}$. Therefore,
$X'\in\Ascr$.

Now, arguing by the contrary, suppose that $X'$ is not stable and consider a blocking edge $e=wf$ for $X'$, i.e. $e$ is interesting for $w$ under $X'_w$, and for $f$ under $X'_f$. Note that $e\notin X$. (For otherwise, $e\in X$ and $e\notin X'$ would imply $e\in L^-\cap E_f=:B_f$, and applying Lemma~\ref{lm:a1-ak} to $A_f:=L^+\cap E_f$ and $B':=B_f-e$, we would have $C_f(X'_f+e)=C_f((X_f-B')\cup A_f)=X'_f$, contrary to the fact that $e$ is interesting for $f$ under $X'_f$.)

Suppose that the given edge $e$ is not not interesing for $f$ under $X$. Then $C_f(X_f+e)=X_f$, whence (using the plotianness~\refeq{plott}) we obtain
 $$
 C_f(X'_f\cup X_f+e)=C_f(X'_f\cup C_f(X_f+e))=C_f(X'_f\cup X_f)=X'_f.
 $$
On the other hand, taking $Z':=C_f(X'_f+e)$, we have
  $$
  C_f(X'_f\cup X_f+e)=C_f(C_f(X'_f\cup X_f)+e)= C_f(X'_f+e)=Z'.
  $$
Therefore, $Z'=X'_f$. But in view of~\refeq{Ze}(ii), the fact that $e$ is interesting under $X'_f$ must imply $Z'\succ_f X'_f$; a contradiction.

Thus, $e$ is interesting for $f$ under $X$.

Now to come to the crucial contradiction, compare $e$ with the active edge $a_w$ and the edge $e'\in X$ incident to $w$; then $\{a_w\}=X'_w$ and $\{e'\}=X_w$. Since $e$ is interesting for $w$ under $X'$, we have $e>_w a_w$. At the same time, since the edge $e$ is interesting for $f$ under $X$, the edge $e$ cannot be interesting for $w$ under $X$ (in view of the stability of $X$); hence $e'>_w e$. Thus, $e$ satisfies property~\refeq{W-admiss} and is more preferred than the $W$-active edge $a_w$; a contradiction.

So $X'$ admits no blocking edges, and therefore it is stable, as required.
\hfill$\qed$
 \medskip

\noindent\textbf{Proof of Proposition~\ref{pr:XYXp}.} 
In order to find a rotation $L\in\Lscr(X)$ determining a required matching $X'$, we first note that, by~\refeq{stab-mix}(c), the sizes of restrictions of $X$ and $Y$ for each vertex $v\in V$ are equal, i.e. $|X_v|=|Y_v|$. Since $X\ne Y$, there exists a vertex $w\in W$ such that $X_w\ne Y_w$. Then $X_w$ and $Y_w$ consist of the unique edges $x_w$ and $y_w$, respectively, and one holds $x_w>_w y_w$ (since $X\prec Y$ implies $X\succ_W Y$, by the polarity~\refeq{stab-mix}(b)).

The edge $y_w=wf$ must be interesting for the vertex $f$ under $X$. Indeed  $Y_f\succeq_f X_f$ implies $C(Y_f\cup X_f)=Y_f$, where $C:=C_f$. Then, using the plottianness, we have
  $$
  Y_f=C(Y_f\cup X_f)=C((Y_f-y_w)\cup (X_f+ y_w))=C((Y_f-y_w)\cup C(X_f+y_w)).
  $$
Now $y_w\notin (Y_f-y_w)$ implies $y_w\in C(X_f+y_w)$, whence $y_w$ is interesting for $f$ under $X$.

Then $w$ has the active edge $a_w=wg$; it satisfies $x_w>_w a_w\ge_w y_w$. Consider the vertex $g$ and the sets $X_g$ and $Y_g$. The edge $a_w$ is interesting for $g$ under $X$, but not under $Y$ (for otherwise we would have $a_w\ne y_w$ and $a_w>_w y_w$, whence  $a_w$ is interesting under  $Y$ for both vertices $w$ and $g$, contradicting the stability of $Y$). So $C_g(Y_g+a_w)=Y_g$.
 \medskip

\noindent\textbf{Claim.} ~\emph{$C_g(X_g+a_w)=X_g+a_w-b$ for some $b\in X_g$; in addition, this $b$ does not belong to $Y$, and $Z:=X_g+a_w-b$ satisfies $X_g\prec_g Z\preceq_g Y_g$.}
   \medskip

 \begin{proof}
~Since $a_w$ is interesting for $g$ under $X_g$, two cases are possible (cf.~\refeq{Ze}(i)): (a) $C(X_g+a_w)=X_g+a_w$, or~(b) $C(X_g+a_w)=X_g+a_w-b$ for some $b\in X_g$, where $C:=C_g$. In case~(a), if $a_w\in Y$, then
  $$
  |C(X_g+a_w)|=|X_g|+1>|Y_g|=|C(Y_g\cup X_g)|;
  $$
at the same time, $X_g+a_w\subseteq Y_g\cup X_g$, contradicting the cardinal monotonicity. And if $a_w\notin Y$, then
  $$
  C(Y_g\cup X_g+a_w)= C(C(Y_g\cup X_g)+a_w)=C(Y_g+a_w),
  $$
and also $|C(Y_g+a_w)|\ge |C(X_g+a_w)|=|X_g|+1$ (using the cardinal monotonicity for the inclusion $Y_g\cup X_g+a_w\supset X_g+a_w$). Then $C(Y_g+a_w)\ne Y_g$, and therefore, the edge $a_w$ is blocking for $Y$ (taking into account $a_w\ne y_w$); a contradiction.

Thus, case~(b) takes place. Suppose that $b\in Y$. Then $C(Y_g\cup X_g+a_w)=C(Y_g+a_w)$, and also
  \begin{multline*}
  C(Y_g\cup X_g+a_w)=C((Y_g-b)\cup X_g+a_w)=C((Y_g-b)\cup C(X_g+a_w)) \\
  =  C((Y_g-b)\cup(X_g+a_w-b)).
  \end{multline*}
Since the operands in the last union do not contain the element $b$, we have
$b\notin C(Y_g+a_w)$. But then $C(Y_g+a_w)\subseteq Y_g+a_w-b\ne Y_g$, which leads to a contradiction both in case $a_w\in Y$, and in case $a_w\notin Y$
(implying that $a_w$ blocks $Y$). So $b\notin Y$. 

Now the desired comparisons for $X_g$, $Z=X_g+a_w-b$ and $Y_g$ easily follow.
 \end{proof}

Let $b=w'g$. The Claim implies that the edge $b$ is associated (forms a tandem) with $a_w$, and that the vertex $w'$ is incident to the edge $y_{w'}\in Y$ which is different from $b=x_{w'}\in X$. Then $x_{w'}>_{w'} y_{w'}$ (in view of $X\succ_W Y$), and we can apply to the pair $(x_{w'},y_{w'})$ reasonings similar to those applied earlier to the pair $(x_w,y_w)$.

By continuing this process, we obtain an ``infinite'' path consisting of alternating $W$-active and $F$-active edges for $X$. In this path, every pair of consecutive edges incident to a vertex in $F$, say, $e,e'\in E_f$, forms a tandem such that: $e\notin X_f\ni e'$, and the set $Z:=X_f+e-e'$ satisfies $Y_f\succeq_f Z$. Extracting from this path a minimal piece between two copies of the same vertex in $W$, we obtain a cycle $L$ being the rotation determining a required matching $X'$, namely, $X':=(X-L^-)\cup L^+$ (where $L^+$ and $L^-$ are the sets of $W$- and $F$-active edges in $L$, respectively). Here all vertices in $W\cap V_L$ are different, and the construction provides that each vertex $w$ among these satisfies $x_w>_w x'_w\ge_w y_w$. As to the vertices in $F\cap V_L$, any $f$ among these can be passed by the cycle $L$ several times (giving pairwise edge-disjoint tandems); using Lemma~\ref{lm:a1-ak} and the above Claim, one can see that $X_f\prec_f X'_f\preceq_f Y_f$.

This completes the proof of Proposition~\ref{pr:XYXp}. \hfill$\qed\qed$


\section{Poset of rotations}  \label{sec:poset}

In this section we establish important additional properties of rotations which give rise to a construction of poset of rotations.

Let $\Tscr$ be a sequence of matchings $X_0,X_1,\ldots,X_N$, where $X_0$ is stable and each $X_i$ is obtained from $X_{i-1}$ by applying (or \emph{shifting along}) a rotation $L_i\in\Lscr(X_{i-1})$, i.e. $X_i=(X_{i-1}-L^-_i)\cup L^+_i$. Then, by Proposition~\ref{pr:XXp}, all matchings $X_i$ are stable and there holds $X_0\prec X_1\prec\cdots\prec X_N$ (where $\prec=\prec_F$). Such a sequence 
$\Tscr$ is called a \emph{route} going from $X_0$ to $X_N$. The set $\{L_1,\ldots,L_N\}$ is denoted by $\Rscr(\Tscr)$. One can see that
 \begin{numitem1} \label{eq:trassa}
\begin{itemize}
 \item[(i)] the amount of matchings in any route $\Tscr$  does not exceed $|E|$;
 \item[(ii)] for any stable matching $X$, there exists a route going from the minimal matching $\Xmin$ (the worst for $F$ and the best for $W$) to the maximal matching  $\Xmax$ which passes $X$.
  \end{itemize}
  \end{numitem1}
Indeed, if $X_i$ is obtained by applying a rotation to $X_{i-1}$, then for each vertex $w\in W$ where $(X_i)_w$ is different from $(X_{i-1})_w$, the matching edge incident to $w$ becomes less preferred. This gives~(i) (and even the bound $|\Tscr|\le |E|/2$, in view of $|W_L|\ge 2$). In its turn,~(ii) easily follows from Proposition~\ref{pr:XYXp}.

Also, applying Lemma~\ref{lm:a1-ak}, one easily shows that the rotations in the active graph $\Gamma(X)$ commute. More precisely:
  \begin{numitem1} \label{eq:commut}
for any subset $\Lscr'\subset\Lscr(X)$, the matching $X'$ obtained from $X$ by deleting the edges from $\cup(L^-\colon L\in \Lscr')$ and adding the edges from $\cup(L^+\colon L\in\Lscr')$ is stable, and each $L'\in\Lscr(X)-\Lscr$ is a rotation in $\Gamma(X')$; in particular, rotations in $\Lscr(X)$ can be applied in an arbitrary order.
  \end{numitem1}

This enables us to obtain the following important property of invariance of the set of rotations applied in the routes connecting fixed stable matchings (originally a property of this sort was revealed in~\cite{IL} for classical stable marriages and subsequently was shown by a number of authors for more general models of stability).
  \begin{lemma} \label{lm:invar_rot}
Let $X,Y\in\Sscr$ and $X\prec Y$. Then for all routes $\Tscr$ going from $X$ to $Y$, the sets of rotations $\Rscr(\Tscr)$ are the same.
 \end{lemma}
 \begin{proof}
~Let $\Xscr$ denote the set of stable matchings $X'$ such that $X\preceq X'\preceq Y$. We know that the set of routes going from $X'\in\Xscr$ to $Y$ is nonempty (in light of Proposition~\ref{pr:XYXp}). Let us say that a matching $X'\in\Xscr$ is \emph{bad} if there exist two routes $\Tscr,\Tscr'$ going from $X'$ to $Y$ such that $\Rscr(\Tscr)\ne \Rscr(\Tscr')$, and \emph{good} otherwise. One has to show that the matching $X$ is good (when $Y$ is fixed).

Suppose that this is not so, and consider a bad matching $X'\in\Xscr$ that is maximal, in the sense that each matching $Z\in\Xscr$ such that $X'\prec Z\preceq Y$ is already good. In any route going from $X'$ to $Y$, the first matching $Z$  after $X'$ is obtained from $X'$ by applying a certain rotation in $\Lscr(X')$. Therefore, the choice of $X'$ implies that there are two rotations $L,L'\in\Lscr(X')$ such that the matchings $Z$ and $Z'$ obtained from $X'$ by applying $L$ and $L'$ (respectively) are good, but there is a route $\Tscr$ from $X'$ to $Y$ passing $Z$ and a route $\Tscr'$ from $X'$ to $Y$ passing $Z'$ for which $\Rscr(\Tscr)\ne\Rscr(\Tscr')$.

At the same time, since the rotations $L$ and $L'$ commute (cf.~\refeq{commut}), $L$ is a rotation for $Z'$, and $L'$ a rotation for $Z$. Therefore, there are two routes $\tilde\Tscr$ and $\tilde\Tscr'$ going from $X'$ to $Y$ such that $\tilde \Tscr$ begins with $X',Z,Z''$, and $\tilde\Tscr'$ begins with $X',Z',Z''$, after which these routes coincide; here $Z''$ is obtained from $Z$ by applying $L'$, or, equivalently, it is obtained from $Z'$ by applying $L$. Then $\Rscr(\tilde\Tscr)=\Rscr(\tilde\Tscr')$. Since both $Z$ and $Z'$ are good, there must be $\Rscr(\tilde\Tscr)=\Rscr(\Tscr)$ and
$\Rscr(\tilde\Tscr')=\Rscr(\Tscr')$. But then we have $\Rscr(\Tscr)=\Rscr(\Tscr')$; a contradiction.
 \end{proof}

We call a route going from $\Xmin$ to $\Xmax$ \emph{full}. By Lemma~\ref{lm:invar_rot}, for any full route $\Tscr$, the set of rotations $\Rscr(\Tscr)$ is the same; we denote this set by $\Rscr$ (so $\Rscr$ consists of all possible rotations applicable to matchings in $\Sscr$). A known method to compare rotations, originally demonstrated in~\cite{IL}, is suitable for our model as well and allows us to introduce a poset structure on $\Rscr$.
 \medskip

\noindent\textbf{Definition 2.} ~For rotations $R,R'\in\Rscr$, we say that $R$ precedes $R'$ and denote this as $R\lessdot R'$ if in every full route, the rotation
$R$ is applied \emph{earlier} than the rotation $R'$.
 \medskip

This binary relation is transitive and anti-symmetric and determines a partial order on $\Rscr$; we call $(\Rscr,\lessdot)$ the \emph{poset of rotations} (or the \emph{rotational poset}) for $G$. A close relationship between this poset and the stable matchings gives rise to a ``compact representation'' of the lattice $(\Sscr,\prec)$ (in spirit of Birkhoff~\cite{birk} where an arbitrary finite distributive lattice is representable via the lattice of ideals of a poset).

More precisely, for each $X\in\Sscr$, take a route $\Tscr$ from $\Xmin$ to $X$ and denote the set of rotations $\Rscr(\Tscr)$ as $\omega(X)$. This set depends only on $X$, and for $R,R'\in\Rscr$, if $R\lessdot R'$ and $R'\in\omega(X)$, then $R\in\omega(X)$, i.e. $\omega(X)$ is an \emph{ideal} of the poset $(\Rscr,\lessdot)$. The converse is valid as well; moreover, the map $\omega$ gives a lattice isomorphism.

 \begin{prop} \label{pr:omega}
The map $X\mapsto \omega(X)$ establishes an isomorphism between the lattice $(\Sscr,\prec_F)$ of stable matchings and the lattice $(\Iscr,\subset)$ of ideals of the poset $(\Rscr,\lessdot)$ (where the greatest lower and least upper bounds for $I,I'\in\Iscr$ are $I\cap I'$ and $I\cup I'$, respectively).
  \end{prop}
  \begin{proof}
~Consider stable matchings $X,Y\in\Sscr$ and take their meet $M:=X\curlywedge Y$ and join $J:=X\curlyvee Y$ in the lattice $(\Sscr,\prec)$. The key part of the proof consists in establishing the following relations:
  \begin{equation} \label{eq:MJ}
\omega(X)\cap \omega(Y)=\omega(M)\quad \mbox{and} \quad
\omega(X)\cup\omega(Y)=\omega(J);
  \end{equation}
in other words, we will show that $\omega$ determines a homomorphism of these lattices.

Let us show the left equality in~\refeq{MJ}. To this aim, we apply a method from the proof of Proposition~\ref{pr:XYXp}, considering the pairs $M\prec X$ and $M\prec Y$. Let $W_1$ ($W_2$) be the set of vertices $w\in W$ such that $M_w\ne X_w$ (resp. $M_w\ne Y_w$). We assert that $W_1\cap W_2=\emptyset$.

To show this, acting as the proof of Proposition~\ref{pr:XYXp}, we choose an arbitrary initial vertex $w\in W_1$ such that $m_w>_w x_w$ (where $\{m_w\}=M_w$ and $\{x_w\}=X_w$), and construct the corresponding alternating path $P$ in the active graph $\Gamma(M)$, starting with $w$ and finishing as soon as we get into an earlier passed vertex, thus obtaining some rotation $L\in\Lscr(M)$. Then for $M':=(M-L^-)\cup L^+$, we have $M\prec M'\preceq X$. Note that the path $P$ and rotation $L$ are constructed uniquely; they are determined by only the graph $\Gamma(M)$ and the initial vertex $w$ (not depending on $X$ in essence).

In case $W_1\cap W_2\ne\emptyset$, taking a common vertex $w\in W_1\cap W_2$, we would obtain the same rotation for both pairs $(M,X)$ and $(M,Y)$. But then the meet for $X$ and $Y$ should be greater than $M$ (namely, at least $M'$ as above); a contradiction.

Therefore, $W_1\cap W_2=\emptyset$. Continuing a construction for the pair $(M,X)$ in a similar way, we obtain a route $\Tscr$ from $M$ to $X$ such that for each intermediate matching $\tilde M$, the set of vertices $w\in W$ with $\tilde M_w\ne X_w$ is a subset of $W_1$. And similarly for a route $\Tscr'$ from $M$ to $Y$ and the set $W_2$. It follows that $\Rscr(\Tscr)\cap\Rscr(\Tscr')=\emptyset$. Now since $\omega(X)=\omega(M)\cup\Rscr(\Tscr)$ and $\omega(Y)=\omega(M)\cup\Rscr(\Tscr')$, we obtain the desired left equality in~\refeq{MJ}.

A proof of the right equality in~\refeq{MJ} is symmetric (this can be conducted by replacing $\prec_F$ by $\prec_W$ and reversing the rotational transformations).

Now the isomorphism of the lattices can be concluded from the fact that for rotations $R,R'\in\Rscr$, the relation $R\lessdot R'$ does not hold if and only if there exists a stable matching $X$ for which $R'\in\omega(X)\not\ni R$. (Roughly speaking, the lattice $(\Sscr,\prec)$ cannot be ``poorer'' than $(\Iscr,\subset)$.)
  \end{proof}

As a consequence, the set $\Sscr$ is in bijection with the set $\Ascr$ of anti-chains of $(\Rscr,\lessdot)$. (Recall that an \emph{anti-chain} of a poset $(P,<)$ is a set $A$ of pairwise incomparable elements. It determines the ideal $\{p\in P\colon\, \exists\, a\in A\,\vert\, p\le a\}$, and in turn is determined by an ideal, being the set of maximal elements of the latter.)


\section{Constructions}  \label{sec:build}

As is said in~\refeq{trassa}(i),  the number $|\Rscr|$ of rotations does not exceed the number $|E|$ of edges of the graph $G$; therefore, the rotational poset (as a graph) has size $O(|E|^2)$. (In addition, one can see that the sets of edges of the same sign in rotations are pairwise disjoint, which implies that the sum of the numbers of edges in rotations is at most $2|E|$.) In light of Lemma~\ref{lm:invar_rot}, the task of constructing the set of rotations $\Rscr$ is rather straightforward if the initial (least preferred for $F$) stable matching $\Xmin$ is available; namely, it suffices to compose an arbitrary full route, i.e. going from $\Xmin$ to $\Xmax$ (details will be specified later). A less trivial task is to find the preceding relations $\lessdot$ on pairs of rotations, which we defined implicitly via examination of the set of full routes (see Definition~2 in Sect.~\SEC{poset}). This section starts with an efficient method of finding these relations, after which efficient constructions of other relevant objects will be discussed.


\subsection{Finding the generating graph for $(\Rscr,\lessdot)$.}
\label{ssec:graph_H}

We say that a rotation $R\in\Rscr$ \emph{immediately precedes} a rotation $R'\in\Rscr$ if $R\lessdot R'$ and there is no $R''\in\Rscr$ such that $R\lessdot R''\lessdot R'$. Let $H=(\Rscr,\Escr)$ be the directed graph whose edge set $\Escr$ is formed by all pairs $(R,R')$ where $R$ immediately precedes $R'$. In other words, $H$ is the Hasse diagram of the poset $(\Rscr,\lessdot)$; the graph $H$ determines this poset via the natural reachability by directed paths and is useful to work with some applications as well.

An efficient construction of the graph $H$ is based on the following simple fact.

  \begin{lemma} \label{lm:immed_preced}
For $R\in\Rscr$, let $\Imax_{-R}$ denote the maximal ideal of the poset $(\Rscr,\lessdot)$ not containing $R$. Let $I':=\Imax_{-R}\cup\{R\}$. Then the rotation $R$ immediately precedes a rotation $R'$ if and only if $R'$ is a minimal element of the poset not contained in $I'$.
 \end{lemma}
 \begin{proof}
Note that $\Imax_{-R}$ is the complement to $\Rscr$ of the set (``filter'') $\Phi_R$ formed by the rotations $\tilde R$ greater than or equal to $R$ (i.e. $R\lessdot \tilde R$ or  $R=\tilde R$). Clearly the minimal elements of $\Phi_R-\{R\}$ are exactly the rotations $R'$ for which $R$ is immediately preceding. These $R'$ constitute the set of minimal elements not contained in the ideal $I'$.
 \end{proof}

For $R\in\Rscr$, let $\Iscr^+_R$ denote the set of rotations immediately succeeding  $R$. Based on Proposition~\ref{pr:omega} and Lemma~\ref{lm:immed_preced}, we can efficiently construct the set $\Iscr^+_R$ as follows (provided that the matching $\Xmin$ is known).
  \medskip

\noindent\textbf{Constructing $\Iscr^+_R$.} ~Starting with $\Xmin$, we build, step by step, a route $\Tscr$ using all possible rotations except for $R$. Namely, at each step, for the current matching $X$, we construct the set $\Lscr(X)$ of rotations  associated with $X$ (cf. Sect.~\SSEC{rotation}), choose an arbitrary rotation $L\in\Lscr(X)$ different from $R$ and shift $X$ along $L$, obtaining a new current matching $X'$. In case $\Lscr(X)=\{R\}$, the first stage of the procedure finishes.

\noindent The second stage of the procedure consists in shifting the resulting matching $X$ of the first stage along $R$, obtaining a new matching $X'$. The set $\Lscr(X')$ of rotations associated with $X'$ is just output as the desired set $\Iscr^+_R$.
  \medskip

(It is easy to see that the matching $X$ obtained upon termination of the first stage corresponds to the ideal $I=\Imax_{-R}$, i.e. $\omega(X)=I$. The coincidence of 
$\Iscr^+_R$ with $\Lscr(X')$ is also obvious.)

The ``straightforward'' algorithm using the above procedure consists of $|\Rscr|$ \emph{big iterations}, each handling a new rotation $R$ in the list of all rotations. (Note that there is no need to form this list in advance, as it can be formed/updated in parallel with carrying out big iterations. The list is initiated with the set $\Lscr(\Xmin)$.) At each big iteration, the following basic problem is solved:
  \begin{itemize}
\item[(P):] given $X\in\Sscr$, construct the active graph $\Gamma(X)$ and extract in it the set $\Lscr(X)$.
  \end{itemize}

To construct $\Gamma(X)$, we scan the vertices $w\in W$, and for each $w$, we examine the edges $wf\in E_w$ in the order of decreasing their preferences $>_w$, starting with the edge next to $x_w$. For each edge $e=wf$, we decide whether it is interesting for $f$ under $X$ or not (by finding $C_f(X_f+e)$ and comparing it with $X_f$), and the first interesting edge (if exists) is assigned to be the $W$-active edge $a_w$. Along the way, the $F$-active edges and tandems are constructed. This gives the admissible graph $D(X)$. Applying the Cleaning procedure to $D(X)$ and decomposing the obtained active graph $\Gamma(X)$ into rotations are routine and can be performed in linear time $O(|E|)$. As a result,
 \begin{numitem1} \label{eq:solvP}
problem (P) reduces to performing $O(|E|)$ standard operations plus $O(|E|)$ calls to ``oracles'' $C_f$ ($f\in F$).
 \end{numitem1}
(Hereinafter we assume that each choice function $C_f$ is given implicitly, by use of  an ``oracle'' that, being asked of a subset $Z\subseteq E_f$, outputs the value $C_f(Z)$; for simplicity we think of such an operation as spending $O(1)$ ``oracle time''.)

A naive construction of the set $\Iscr^+_R$ for a fixed $R$ reduces to independent solutions of $O(|E|)$ instances of problem (P) (each concerning an element of an arising route $\Tscr$);  for this reason, the time bounds exposed in~\refeq{solvP} should be multiplied by $O(|E|)$. However, this procedure can be accelerated. To do so, for each vertex $w\in W$, we store the last edge $e=wf$ examined during previous steps. If this edge was active and took part in an already applied rotation, then the further examination in $E_w$ should be started with the edge next to $e$. And if $e$ was not used in previous rotations, then it continues to be active at the current step. This is justified by use of Lemma~\ref{lm:a1-ak}, which implies that if a rotation $L$ was contained in the current collection $\Lscr$ of rotations at some step but was not applied at that moment, then it continues to stay in $\Lscr$ on subsequent steps, until becomes applied.

This gives an improved procedure of constructing $\Iscr^+_R$, which examines each edge in $E_w$ at most once, and therefore, the whole work with one rotation $R\in\Rscr$ is performed in time $O(|E|)$ (including ``oracle calls''). Then, summing over $\Rscr$, we obtain the following

  \begin{prop} \label{pr:time_for_H}
When $\Xmin$ is available, the minimal generating graph $H=(\Rscr,\Escr)$ for the poset $(\Rscr,\lessdot)$ is constructed in time $O(|E|^2)$ (including ``oracle calls'').
 \end{prop}


\subsection{Finding the initial stable matching $\Xmin$.} \label{ssec:Xmin}

In order to construct $\Xmin$, one can use a method from~\cite[Sec.~3.1]{AG} intended for a wider class of models of stability in bipartite graphs. Below we describe its specification to our model CBM. An alternative method, based on the classical idea of ``deferred acceptance'', can be found in~\cite[Sec.~4.1]{FZ}.

The algorithm iteratively constructs triples of sets $(B^i,X^i,Y^i)$, $i=0,1,\ldots,i,\ldots$. Initially, one puts $B^0:=E$. In the input of a current, $i$-th, iteration, there is a set $B^i\subseteq E$ (already known), which is transformed on two stages of the iteration.
 \smallskip

On \emph{1st stage} of $i$-th iteration, $B^i$ is transformed into $X^i\subseteq E$ by applying the CF $C_w$ to each restriction $B^i_w={B^i}\rest{E_w}$ ($w\in W$), i.e. $X^i$ is the matching satisfying $X^i_w=C_w(B^i_w)$ for all $w\in W$. In other words, in our model with linear preferences and unit quotas in $W$, for each vertex $w\in W$, in the set $B^i_w$, we take the most preferred (w.r.t. $>_w$) edge $e$ and put $X^i_w:=\{e\}$. In case $B^i_w=\emptyset$, one puts $X^i_w:=\emptyset$ (and the vertex $w$ becomes deficit). Thus, $X^i$ obeys the quotas for all vertices of the part $W$.
 \smallskip

On \emph{2nd stage} of $i$-th iteration, the matching $X^i$ is transformed into $Y^i\subseteq E$ by applying the CF $C_f$ to each restriction $X^i_f= {X^i}\rest{E_f}$ ($f\in F$), i.e. $Y^i$ is the matching satisfying $Y^i_f=C_f(X^i_f)$ for all $f\in F$. (Therefore, $Y^i$ is acceptable for all vertices; however, it need not be stable.)

The obtained sets $X^i,Y^i$ are then used to generate the input set $B^{i+1}$ for the next iteration. More precisely:
  \begin{numitem1} \label{eq:Bi+1}
~$B^{i+1}$ consists of all $e\in E$ such that $e\in B^i$, and either $e\in
X^i\cap Y^i$, or $e\notin X^i\cup Y^i$ (in other words, $e\in B^i$ does not get into $B^{i+1}$ if $e$ is in $X^i$ but not in $Y^i$).
  \end{numitem1}

Then $B^{i+1}$ is transformed into $X^{i+1}$ and $Y^{i+1}$ as described above, and so on. The process terminates when at the current, $p$-th say, iteration, we obtain the equality $Y^p=X^p$ (equivalently, $B^{p+1}=B^p$). One can see that $B^0\supset B^1\supset\cdots\supset B^p$, implying that the process is finite and the number of iterations does not exceed $|E|$.

By a result proved in~\cite{AG}, the following takes place.
 \begin{prop} \label{pr:Xp}
The resulting matching $X^p$ is stable and optimal for $W$, i.e. $X^p=\Xmin$.
 \end{prop}

It should be noted that for a general model considered in~\cite{AG}, which can deal with real-valued functions on $E$, the process of transforming the corresponding functions $(b^i,x^i,y^i)$ can be infinite, though it always converges to some triple $(\hat b,\hat x,\hat y)$ satisfying $\hat x=\hat y$; one proves (in Theorems~1 and~2 in~\cite{AG}) that the limit function $\hat x$ is stable and optimal for the appropriate part of vertices. In our particular boolean case, we simply obtain Proposition~\ref{pr:Xp}.

Next, observe that at $i$-th iteration of the algorithm, the set $X^i$ can be formed from $B^i$ in $O(|W|)$ time, by taking the first elements in the restrictions $B^i_w$, $w\in W$. The sets $X^i_f$ for all $f\in F$ can also be constructed in time $O(|W|)$, and the operator $C_f$ is applied to each $X^i_f$ at most once (to compute $Y^i_f$). In light of this, the algorithm can be implemented so as to obtain the following estimate:
  \begin{numitem1} \label{eq:time_Xmin}
the matching $\Xmin$ can be constructed in time $O(|E| |V|)$ (including ``oracle calls'').
  \end{numitem1}

This together with Proposition~\ref{pr:time_for_H} gives the following result.

  \begin{theorem} \label{tm:time_poset}
The poset of rotations $(\Rscr,\lessdot)$ can be constructed in time $O(|E|^2)$ (including ``oracle calls'').
 \end{theorem}

\noindent\textbf{Remark 2.} The above constructions of rotations and their poset can be extended to CBM with a graph $G=(V,E)$ containing  multiple edges. To do so, we use the reduction to a graph $G'=(V',E')$ without multiple edges obtained by replacing edges $e$ by the subgraphs $K_e$ as described in Remark~1 (in Sect.~\SEC{def}); here the set $U$ of edges to be replaced contains at least one edge from any two multiple edges. Let us say that stable matchings in $G'$ are \emph{close} if they differ only within the cycles $O_e$ for some $e\in U$ (taking into account case~($\ast$) mentioned in Remark~1). Then the set $\Sscr$ of stable matchings for $G$ is isomorphic to the set of classes of closeness for $G'$. According to this, the set  $\Rscr'$ of rotations for $G'$ is partitioned into two subsets $\Rscr'_1$ and $\Rscr'_2$, where each rotation in $\Rscr'_2$ corresponds to a cycle $O_e$ for some $e\in U$ (and $\Rscr'_1$ has no rotation of this sort). One can see that if at some moment of constructing a route for $G'$, one applies a rotation containing an edge from $K_e$ for some $e\in U$, then the matching becomes stabilized within $K_e$, i.e. no subsequent rotation can use edges from $K_e$. (This follows from the structure of $K_e$ and the fact that if an edge from the current matching is deleted by applying a rotation, then subsequently this edge cannot enter the matching again.)

As a consequence, each rotation in $\Rscr'_2$ is a maximal element of the poset $(\Rscr',\lessdot')$. In turn, the rotations for $G$ one-to-one correspond to the elements of $\Rscr'_1$ (being the images of the latter under the replacement of the subgraphs $K_e$ by the edges $e$). So the rotational poset $(\Rscr,\lessdot)$ for $G$ is isomorphic to the restriction of the poset $(\Rscr',\lessdot')$ on the set $\Rscr'_1$. Since $|E'|< 8|E|$, the poset for $G$ is constructed in time $O(|E|^2)$, in view of Theorem~\ref{tm:time_poset}.


\section{Affine representation and stable matchings of minimum cost}  \label{sec:affine}

The bijection $\omega$ figured in Proposition~\ref{pr:omega} enables us to show an affine representation of the lattice $(\Sscr,\prec_F)$, by an analogy with that done in~\cite{FZ} for the general boolean stability problem, or in~\cite{MS} for the stable allocation problem.

Recall that any rotation $R\in\Rscr$ arises as a cycle of a certain active graph $\Gamma(X)$, and the set of its edges has a fixed partition denoted as  $(R^+,R^-)$, formed by  $W$- and $F$-active edges, respectively. We associate with $R$ the characteristic vector  $\beta^R\in\Rset^E$ taking value 1 for $e\in R^+$, \,$-1$ for $e\in R^-$, and 0 for the remaining edges.

Besides $\Rset^E$, we also will deal with the space $\Rset^\Rscr$ whose coordinates are indexed by the rotations; in this case we denote the unit base vector related to a rotation $R$ as $\langle R\rangle$.

Also for a subset (matching) $X\subseteq E$, its characteristic 0,1-vector $\Rset^E$ is denoted as $\chi^X=\chi_E^X$, and similarly for a subset $I\subseteq\Rscr$, its characteristic 0,1-vector in $\Rset^\Rscr$ is denoted as $\chi^I=\chi^I_\Rscr$. 

Let $A\in \Rset^{E\times\Rscr}$ be the matrix whose columns are formed by the vectors $\beta^R$, $R\in\Rscr$. Proposition~\ref{pr:omega} implies that 
  \begin{numitem1} \label{eq:affine}
the correspondence $\lambda\in\Rset^\Rscr\stackrel{\gamma}{\longmapsto} x\in\Rset^E$, where $x$ is defined as $\chi^{\Xmin}+ A\lambda$, generates a bijection between the characteristic vectors $\chi_\Rscr^I$ of ideals $I$ of the poset $(\Rscr,\lessdot)$ and the characteristic vectors $\chi_E^X$ of stable matchings $X\in\Sscr$.
 \end{numitem1}
(Here $\gamma$ corresponds to a natural extension of the map $\omega^{-1}$ in Proposition~\ref{pr:omega}.)
 
By reasonings in Sect.~\SEC{build}, the matrix $A$ can be constructed efficiently, in time $O(|E|^2)$ (including ``oracle calls''). Due to this, the affine representation of the lattice $(\Sscr,\prec_F)$ exhibited in~\refeq{affine} (where the order $\prec_F$ is agreeable with the inclusion $\subset$ in the lattice of ideals for $(\Rscr,\lessdot)$) can be useful to reduce certain problems on stable matchings to ``simpler'' ones on ideals of the rotational poset.
 
First of all, notice that the affine map $\gamma$ in~\refeq{affine} sends the convex hull $\convex(\Iscr)\subset \Rset^\Rscr$ of the set $\Iscr$ of ideals of the rotational poset onto the convex hull $\convex(\Sscr)\subset \Rset^E$ of the set $\Sscr$ of stable matchings. (Here to simplify the description, we identify ideals of the poset as well as stable matchings with their characteristic vectors.) The polytope $\convex(\Iscr)$ is full dimensional in $\Rset^\Rscr$ (since each unit base vector $\langle R\rangle$ can be expressed as the difference of characteristic vectors of two ideals in $\Iscr$). It is described by the inequalities
  \begin{eqnarray}
  0\le\lambda(R)\le 1,\;\; && R\in\Rscr;  \label{eq:conv72} \\
  \lambda(R)\ge \lambda(R'),\;\; && R,R'\in\Rscr,\;  R\lessdot R'. \label{eq:conv73}
   \end{eqnarray} 
   
This linear system matches a description of the so-called \emph{order polytope} $\Pscr_Q$ of a finite poset $Q$ introduced by Stanley~\cite{stan}; in our case $Q=(\Rscr,\lessdot)$ and $\Pscr_Q=\convex(\Iscr)$. The hyperfaces (facets) in $\convex(\Iscr)$ are expressed by the inequatities of three types: 
  \begin{numitem1} \label{eq:facets}
(a) $\lambda(R)\le 1$ for minimal $R$; (b) $\lambda(R)\ge 0$ for maximal $R$; and (c) $ \lambda(R)\ge \lambda(R')$ if $R$ immediately precedes $R'$.
 \end{numitem1}
 
As to vertices, we observe that
 \begin{numitem1} \label{eq:vertex_ord}
the vertices of $\convex(\Iscr)$ one-to-one correspond to the ideals in $\Iscr$.
 \end{numitem1}
(Indeed, each ideal $I\in\Iscr$ determines a vertex of $\convex(\Iscr)$, since $\chi^I$ is the unique vector maximizing the inner product $a\chi^{I'}$ over all ideals $I'\in\Iscr$, where $a$ takes value 1 on the rotations $R\in I$, and $-1$ on the rotations $R\in \Rscr-I$. The converse is obvious.)

An important property of the map $\gamma$ in~\refeq{affine} follows from the fact that
  \begin{numitem1} \label{eq:full_rank}
the matrix $A$ has the full column rank.
 \end{numitem1}
 
Such a fact was shown in~\cite[Th.~1.4]{FZ} for a general boolean case (where CFs for all vertices obey axioms (A1),(A2),(A3)). A proof for our model CBM is relatively simple and we can now give it for the completeness of our description.
  \medskip
  
\noindent\textbf{Proof of~\refeq{full_rank}.}
Let $N:=|\Rscr|$. Number the rotations as $R(1),\ldots,R(N)$ following the rule:  $R(i)\lessdot R(j)$ implies $i<j$ (such a numbering is known as a ``topological sorting'' of the vertices of an acyclic digraph). For each $i=1,\ldots,N$, choose in $R(i)$ one edge in the ``negative'' part $R(i)^-$ and denote it as $e_i$. Then the obtained numerations satisfy:
  \begin{numitem1} \label{eq:numerat}
(a) all edges $e_1,\ldots,e_N$ are different; and (b) for any $1\le i<j\le N$, the value of $\beta^{R(j)}$ on $e_i$ is zero.
  \end{numitem1}
 
These properties follow from the fact that for a rotation $R$ occurring in the active graph $\Gamma(X)$ for a matching $X$, in each vertex $w\in W$ contained in $R$, the matching edge $e=wf\in X$ is preferred to the active edge $a=wf'$, i.e. $e>_w a$. We have $e\in R^-$ and $a\in R^+$, and by applying the rotation $R$ to $X$, the edge $e$ of the current matching is replaced by the less preferred edge $a$. Therefore, in any full route, after the usage of the rotation $R$, none of subsequent rotations could contain the edge $e$. This implies both properties in~\refeq{numerat}. (In fact, one can observe the following: any edge $e$ belongs to at most two rotations, and if $e$ belongs to $R(k)$ and $R(i)$ with $k<i$, then these rotations are comparable as $R(k)\lessdot R(i)$, and there holds $e\in R(k)^+$ and $e\in R(i)^-$.)

Now rearrange the orders of columns and rows of the matrix $A$ so that the columns $\beta^{R(i)}$ follow by increasing the indexes $i$ (from left to right), and the first $N$ rows correspond to the edges $e_1,\ldots,e_N$ following by increasing their indexes (from top to bottom). Then~\refeq{numerat} implies that the submatrix formed by the first $N$ rows is lower triangular, with coefficients $-1$ on the diagonal. This gives~\refeq{full_rank}. \hfill$\qed$
  \medskip
  
From~\refeq{affine} and~\refeq{full_rank} we obtain 

  \begin{corollary} \label{cor:congruent}
$\dim(\convex(\Sscr))=\dim(\convex(\Iscr))=|\Rscr|$, and the polytope $\convex(\Sscr)$ is affinely congruent to the order polytope $\convex(\Iscr)$ of the poset $(\Rscr,\lessdot)$.
  \end{corollary}
  
In particular, the vertices of $\convex(\Sscr)$ are formed by the stable matchings in $\Sscr$ and one-to-one correspond to the ideals in $\Iscr$, the polytope $\convex(\Sscr)$ has $O(|\Rscr|^2)$ facets, all of them are the images by $\gamma$ of the facets of $\convex(\Iscr)$ (pointed out in~\refeq{facets}) and can be efficiently listed.
 \smallskip
 
In the rest of this section, we consider the \emph{problem of finding a stable matching of minimum cost}:
 \begin{numitem1} \label{eq:min-cost}
given \emph{costs} $c(e)\in\Rset$ of edges $e\in E$, find a stable matching  $X\in\Sscr$ having the minimum total cost $c(X):=\sum(c(e)\colon e\in E)$.
  \end{numitem1}
  
Note that since all stable matchings $X$ have the same size $|X|$, the function $c$ can be given up to a constant; in particular, it can be positive. Replacing $c$ by $-c$, we obtain an equivalent problem of maximizing $c(X)$ over $X\in\Sscr$. 

A nice method elaborated in~\cite{ILG} to minimize a linear function over stable marriages and subsequently successfully applied by a number of authors to some other models of stability consists in a reduction to a linear minimization problem on ideals of a related poset (when such a poset exists and can be constructed efficiently), and then the latter problem is reduced to the standard minimum cut problem, by a method due to Picard~\cite{pic}. To make our paper more self-contained, we briefly explain how such an approach works in our case, to efficiently solve~\refeq{min-cost} (below we prefer to follow a description as in the work~\cite{karz}).

In the beginning we compute the cost $c\beta^R$ ($=c(R^+)-c(R^-)$) of each rotation $R\in\Rscr$. Then for each $X\in\Sscr$ and the corresponding ideal $I:=\omega(X)$, we have
  $$
  c(X)=c(\Xmin)+\sum(c\beta^R\colon R\in I).
  $$

This leads to the following equivalent problem dealing with rotations $R$ endowed with weights (or costs) $\zeta(R):=c\beta^R$:
  \begin{numitem1} \label{eq:min-cost-ideal}
in the poset $(\Rscr,\lessdot)$, find an ideal $I\in\Iscr$ minimizing the total weight $\zeta(I):=\sum(\zeta(R)\colon R\in I)$.
 \end{numitem1}
 
It is convenient to slightly extend the setting of the latter problem, by considering an arbitrary finite directed graph $H=(V_H,E_H)$ and a weight function $\zeta: V_H\to \Rset$. It is required to find a closed set of vertices $X\subseteq V_H$ having the minimal weight $\zeta(X)$; we call this \emph{problem~$(\ast)$}. Here a set $X$ is called \emph{closed} if $H$ has no edges going from $V_H-X$ to $X$. (Clearly, in the graph of a poset, the closed sets are exactly the ideals.)

Following~\cite{pic}, problem~$(\ast)$ reduces to the minimum cut problem in the directed graph $\hat H=(\hat V,\hat E)$ with the edge capacity function $h$ that are obtained from $H,\zeta$ by: 

(a) adding two vertices: ``source'' $s$ and ``sink'' $t$; 

(b) adding the set $E^+$ of edges $(s,v)$ for all vertices $v$ contained in $V^+:=\{v\in V_H\colon \zeta(v)>0\}$;

(c) adding the set $E^-$ of edges $(u,t)$ for all vertices $u$ contained in $V^-:=\{u\in V_H\colon \zeta(u)<0\}$;

(d) assigning the capacities $h(s,v):=\zeta(v)$ for $v\in V^+$, \; $h(u,t):=|\zeta(u)|$ for $u\in V^-$, and $h(e):=\infty$ for all $e\in E_H$.

Recall that by an $s$--$t$ \emph{cut} in $\hat H$ one is meant the set of edges $\delta(A)$ going from a subset of vertices $A\subset \hat V$ satisfying $s\in A\not\ni t$ to its complement $\hat V-A$; the capacity of this cut is defined to be the value $h(\delta(A)):=\sum(h(e)\colon e\in\delta(A))$. One can see that $\delta(A)$ has the minimum capacity among all $s$--$t$ cuts if and only if $X:=V_H-A$ is a closed set with the minimum weight $\zeta(X)$ in $H$.

Indeed, for an $s$--$t$ cut $\delta(A)$, the value $h(\delta(A))$ is finite if and only if $\delta(A)$ contains no edges from $H$ (in view of the infinite capacity of such edges). It follows that $\delta(A)\subseteq E^+\cup E^-$ and that the set $X$ is closed. Then
   \begin{multline*}
h(\delta(A))=h(\delta(A)\cap E^+)+h(\delta(A)\cap E^-) \\
= \zeta(X\cap V^+)+\sum(|\zeta(u)|\colon u\in (V_H-X)\cap V^-)\\
  =\zeta(X\cap V^+)+ \zeta(X\cap V^-)-\zeta(V^-)=\zeta(X)-\zeta(V^-).
  \end{multline*}

Thus, $\zeta(X)$ differs from $h(\delta(A))$ by the constant $\zeta(V^-)$, whence the desired property follows, and we can conclude with the following
 \begin{theorem} \label{tm:min_cost}
The problem on finding a stable matching of minimum cost~\refeq{min-cost} for the model CBM in question is solvable in strongly polynomial time (estimating the number of standard operations and ``oracle calls'').
  \end{theorem}


\section{Arbitrary quotas on $W$}  \label{sec:general}

Up to now we have described constructions and proved assertions in the assumption that in the model of stability that we deal with (CBM) the quotas of all vertices in the part $W$ are equal to 1. In fact, the obtained results can be extended, rather straightforwardly, to a general case of arbitrary quotas $q(w)\in\Zset_+$, $w\in W$, and below we give a brief outline of needed refinements and changes, leaving to the reader verification of details where needed.
\smallskip

1) First of all, we should specify the construction of active graph $\Gamma(X)$ for a stable matching $X\subseteq E$ (Sect.~\SSEC{active_g}). Earlier, for a fully filled vertex $w\in W^=$, we denote by $x_w$ the unique edge in $X_w=X\cap E_w$. Now, considering a fully filled vertex $w\in W^=$ (i.e. satisfying $|X_w|=q(w)$), we denote by $x_w$ the last (least preferred) edge in $X_w$. The definitions of $W$-admissible edge $a_w$ in $E_w$ and tandem $(a_w,b_f^w)$ remain as in~\refeq{W-admiss} and~\refeq{ab}. So, as before, each vertex in $W$ has at most one $W$-admissible edge.
\smallskip

2) As before, in the graph $D=D(X)$ induced by the directed $W$- and $F$-admissible edges, for a vertex $f\in F\cap V_D$, the number of entering $W$-admissible edges (of the form $a_w=(w,f)$) is greater than or equal to the number of leaving $F$-admissible edges (of the form $b=(f,w')$). (Here the former number exceeds the latter one when some entering edge $a_w=(w,f)$ satisfies $C_f(X_f+a_w)=X_f+a_w$ (and therefore, $a_w$ does not create a tandem) or when there are two tandems $(a,b),(a',b')$ with $b=b'$). At the same time, for a vertex $w\in W\cap V_D$, there is one leaving $W$-admissible edge, namely, $a_w$, but the number $\sigma_w$ of entering $F$-admissible edges $(f,w)$ may be 0, 1 or more (the latter can happen when $q(w)>1$).
 \smallskip
 
3) The Cleaning procedure remains literally the same, it transforms $D$ into the active graph $\Gamma=\Gamma(X)$. Note that for vertices $w$ in $W_\Gamma:=W\cap V_\Gamma$, the procedure provides validity of $|\deltain(w)|\ge|\deltaout(w)|=1$ (where $\deltain(w)$ and $\deltaout(w)$ denote the sets of entering and leaving edges incident to $w$ in $\Gamma$, respectively); in case $\sigma_w>1$, apriori a strict inequality is admitted here. However, this does not happen in reality, and property~\refeq{balance} follows from easy balance relations (since, as before, $|\deltain(v)|\ge|\deltaout(v)|$ holds for all $v\in V_\Gamma$). Thus, as before, $\Gamma$ is split into pairwise edge disjoint cycles-rotations, each vertex in $W_\Gamma$ belongs to exactly one rotation, and any rotation $L$ passes each vertex of $W_L$ exactly once, though may pass the same vertex of $F_L$ many times. 
 \smallskip
 
4) Lemma~\ref{lm:a1-ak} remains valid and its proof does not change. In the proof of Proposition~\ref{pr:XXp}, now the set $X'_w$ need not consist of only one edge, but according to the general rule, is defined as $X'_w:=X_w-e+a_w$, where $\{e\}=L^-\cap E_w$. This does not affect the proof, up to minor corrections. In the proof of Proposition~\ref{pr:XYXp}, in case $X_w\ne Y_w$, instead of $\{x_w\}=X_w$ and $\{y_w\}=Y_w$ (as before), the edges $x_w$ and $y_w$ should now be chosen from $X_w-Y_w$ and $Y_w-X_w$, respectively, so as to satisfy $x_w>_w y_w$ (which is possible to do since $X_w\succ_w Y_w$). The structure of the proof and the principal moments in it preserve.
 \smallskip
 
5) In the proof of Proposition~\ref{pr:omega}, instead of $\{m_w\}=M_w$ and $\{x_w\}=X_w$, we now should choose $m_w\in M_w-X_w$ and $x_w\in X_w-M_w$ so as to satisfy $m_w>_w x_w$. Other details throughout Sect.~\SEC{poset} do not change in essence. 
 \smallskip
 
6) One can check that the constructions and results established in Sects.~\SEC{build} and~\SEC{affine}, preserve for arbitrary quotas on $W$. In particular, Theorem~\ref{tm:time_poset} remains valid, which asserts that the rotational poset $(\Rscr,\lessdot)$ can be constructed in time $O(|E|^2)$.


\section{Model of stability with sequential choices}  \label{sec:sequent}

As is mentioned in the Introduction, our model of stability (CBM) originally appeared by a reduction from a more general model on stable matchings in a bipartite graph. In the latter, the preferences of agents of one side (``firms'') are given by use of plottian and cardinally monotone CFs, whereas those of the other side (``workers'') by use of so-called sequential CFs; for definiteness we will refer to this model as \emph{sequential}, or the \emph{S-model}. The work~\cite{dan} establishes a reduction of an instance of S-model to an instance of CBM under which the sets of stable matchings turn out to be isomorphic. Below, following~\cite{dan}, we review a description of the S-model and its reduction to CBM and recall the result on isomorphisms of the corresponding stable matchings and their lattices. Then we discuss what corollaries for the S-model can be derived from our results and applications obtained for CBM. 

It should be noted that~\cite{dan} considers models with sequential CFs in a wider context, however we will be interested only in the one directly related to a reduction to CBM that we deal with. For simplicity we assume that the graphs in question have no multiple edges (except for a remark in the end of this section).

As before, one considers a bipartite graph $G=(V,E)$ with vertex parts $W$ (``workers'') and $F$ (``firms''), and for each vertex $f\in F$, one is given a plottian and cardinally monotone CF $C_f:2^{E_f}\to 2^{E_f}$. At the same time, for each vertex $w\in W$, one is given a sequence of  \emph{linear} CFs $C_w^1,\ldots,C_w^{q(w)}$ on $E_w$; this means that each CF $C_w^i$ corresponds to a linear order $>_w^i$ on $E_w$ and picks from each subset $Z\subseteq E_w$ the maximal element relative to $>_w^i$. One may assume that $q(w)\le |E_w|$.
  \medskip
  
\noindent\textbf{Definition.} Let $C_w^1\ast\cdots\ast C_w^{q(w)}$ denote the CF $C_w$ that for each $Z\subseteq E_w$, determines the subset $C_w(Z)$ consisting of the elements $z_1,\ldots,z_k$ (where $k=\min\{|Z|,q(w)\}$) assigned by the following recursive rule: $z_i$ is the maximal element relative to $>_w^i$ in the set  $Z-\{z_1,\ldots,z_{i-1}\}$. Such a  $C_w$ is called the \emph{sequential} CF of rank $q(w)$ generated by linear orders $>_w^1,\ldots,>_w^{q(w)}$ (or linear CFs $C_w^1,\ldots,C_w^{q(w)}$). 
  \medskip
  
One shows that this $C_w$ is plottian and quota-filling, with the quota $q(w)$. A collection $\{C_v,\; v\in V\}$ of CFs as above just gives rise to what we call an instance of S-model. Note that this model is a particular case of SBM (special boolean model mentioned in the Introduction), and at the same time, it generalizes CBM. (Note also that, as is said in~\cite{dan}, it is known that not every plottian quota-filling CF could be represented as a sequential one, in the above sense.)

Next we start describing the reduction of the S-model with CFs $C_v$ ($v\in V$) as above. The graph $G$ is transformed by replicating vertices of the part $W$. More precisely,
  \begin{numitem1} \label{eq:repl_w}
each vertex $w\in W$ is replaced by $q(w)$ vertices $w^1,\ldots, w^{q(w)}$, and accordingly, each edge $wf\in E$ generates $q(w)$ edges $w^if$, $i=1,\ldots, q(w)$.
  \end{numitem1}
  
Denote the obtained graph by $\tilde G=(\tilde V,\tilde E)$ and denote by $\pi$ the natural map (projection) of $\tilde V\cup\tilde E$ to $V\cup E$. As a rule (though not always), the objects related to $\tilde G$ will be denoted with tildes. In particular, copies of vertices $f\in F$ in $\tilde G$ can be denoted by $\tilde f$, and for any vertex $\tilde v\in \tilde V$, the set of its incident edges in $\tilde G$ can be denoted by $\tilde E_{\tilde v}$. 

Now we explain how to define the preferences and choice functions for vertices $\tilde v\in\tilde V$, denoting the latter as $\tilde C_{\tilde v}$. This is routine for the vertices in $\tilde W$, namely: 
  \begin{numitem1} \label{eq:Cwi}
for a vertex $w^i\in \tilde W$ (where $w\in W$ and $1\le i\le q(w)$), $\tilde C_{w^i}$ is the linear CF related to the linear order $>_{w^i}$ on $\tilde E_{w^i}$ that is a copy of the order $>_w^i$ on $E_w$.
  \end{numitem1}
  
For vertices in $\tilde F$, the construction of CFs is less trivial. More precisely, for $\tilde f\in \tilde F$ (a copy of $f$ in $F$), consider a subset of edges $\tilde Z\subseteq \tilde E_{\tilde f}$ and its image $Z=\pi(\tilde Z)$ in $G$, and form $\tilde C_{\tilde f}$ as follows:
 \begin{numitem1} \label{eq:Cf}
for each $wf\in C_f(Z)$, take the edge $w^if$ in the ``fibre'' $\pi^{-1}(wf)$ such that $w^if$ belongs to the set $\tilde Z$ and has the minimal number $i$; then $\tilde C_{\tilde f}$ is the union of the taken elements.
  \end{numitem1}
  
Let $\Sscr$ denote the set of stable matchings for the S-model with graph $G$ and CFs $C_f$ ($f\in F$) and $C_w$ ($w\in W$) (where $C_f$ is plottian and cardinal monotone, while $C_w$ is sequential and generated by linear orders $>_w^1,\ldots,>_w^{q(w)}$). Let $\tilde\Sscr$ denote the set of stable matchings for CBM with the constructed graph $\tilde G$, CFs $\tilde C_{\tilde f}$ ($f\in F$) and linear orders $>_{w^i}$. In~\cite{dan}, the following key properties are proved:
 \begin{numitem1} \label{eq:red_prop}
 \begin{itemize}
\item[(i)] for each stable matching $\tilde X\in\tilde\Sscr$, the restriction of the map $\pi$ to $\tilde X$ is injective;
  \item[(ii)] $\pi$ induces a bijection between the sets of stable matchings $\tilde\Sscr$ and $\Sscr$;
  \item[(iii)] the correspondence between stable matchings $\tilde X\stackrel{\pi}{\longmapsto} X$ gives an isomorphism of the lattices on $\tilde\Sscr$ and $\Sscr$, i.e. for $X,Y\in\Sscr$, there hold $\pi^{-1}(X\curlywedge Y)=\pi^{-1}(X)\,\tilde\curlywedge\,\pi^{-1}(Y)$ and $\pi^{-1}(X\curlyvee Y)=\pi^{-1}(X)\,\tilde\curlyvee\,\pi^{-1}(Y)$ (where $\curlywedge$ and $\curlyvee$ mean the meet and join operations on $\Sscr$, and similar notation with tildes is used for $\tilde\Sscr$).
   \end{itemize}
   \end{numitem1}
   
Based on~\refeq{red_prop} and using the obtained results on rotations for CBM, we can give a desctiption of rotations for the S-model.

For this purpose, consider a rotation $\tilde R\subset \tilde E$ for CBM with the graph $\tilde G$. It is applied to move from some stable matching $\tilde X\in\tilde S$ to an immediately succeeding (relative to $\prec_{\tilde F}$ in $\tilde S$) stable matching $\tilde X'$, namely, $\tilde X'=(\tilde X-\tilde R^-)\cup \tilde R^+$. Let $X:=\pi(\tilde X)$ and $X':=\pi(\tilde X')$. By~\refeq{red_prop}(i),(ii), $\pi$ establishes a bijection between $\tilde X$ and $X$, as well as between $\tilde X'$ and $X'$, whence one can conclude that $\tilde R^-$ is bijective to $\pi(\tilde R^-)$,  and $\tilde R^+$ is bijective to $\pi(\tilde R^+)$. Apriori we cannot exclude the situation when $\Delta:=\pi(\tilde R^-)\cap \pi(\tilde R^+)\ne\emptyset$ (in this case, there is an edge $wf\in E$ such that $w^if\in\tilde R^-$ and $w^jf\in\tilde R^+$ with $i\ne j$), and at this moment we leave the possibility of $\Delta\ne \emptyset$ as an open question. Define
  $$
  R^-:=\pi(\tilde R^-)-\Delta \quad\mbox{and}\quad R^+:=\pi(\tilde R^+)-\Delta.
  $$
Then $X'=(X-R^-)\cup R^+$, $|R^-|=|R^+|$, and both sets $R^-$ and $R^+$ are nonempty (for otherwise, there would be $\tilde X\ne\tilde X'$ but $\pi(\tilde X)=\pi(\tilde X')$, contrary to~\refeq{red_prop}(ii)).

One can see that $R:=\pi(\tilde R)-\Delta$ generates an edge simple cycle in $G$ (induced by the rotation $\tilde R$ regarded as a cycle), and in this cycle the edges from $R^-$ (``negative'') and $R^+$ (``positive'') alternate. This $R$ (regarded, depending on the context, as a set of edges or as a cycle) just plays the role of rotation in $G$, and we say that the stable matching $X'$ is obtained from $X$ by applying the rotation $R$. 

(We also leave it as an open question whether a rotation $R$ can pass the same vertex in $W$ more than once, which, as we know, is impossible for $\tilde R$ and $\tilde W$.)

Summing up what is said above, we can derive from the properties in~\refeq{red_prop} and results for CBM the following, rather straightforward, corollaries:
  \begin{numitem1} \label{eq:rot_sequent}
  \begin{itemize}
\item[(i)] each stable matching $X\in\Sscr$ can be obtained from the minimal matching in $(\Sscr,\prec_F)$ by applying a sequence of rotations in $G$;
\item[(ii)] the correspondence $\tilde R\mapsto R=\pi(\tilde R)-(\pi(\tilde R^+)\cap \tilde R^-))$ gives a bijection between the rotations in $\tilde G$ and $G$;
\item[(iii)] $\pi$ induces an isomorphism between the rotational posets for $\tilde G$ and $G$.
 \end{itemize}
 \end{numitem1}
   
Note that the size of the graph $\tilde G$ created by replicating each vertex $w\in W$ by $q(w)\le |F|$ copies can be roughly estimated as $O(|W| |F|)$ regarding the vertices and $O(|W| |F|^2)$ regarding the edges. Therefore, from Theorem~\ref{tm:time_poset} one can conclude that
  \begin{numitem1} \label{eq:time_sequent} 
for the S-model with a graph $G=(V=W\sqcup F,E)$, the set of rotations and their poset can be constructed in time $O(|W|^2|F|^4)$ (including ``oracle calls''). 
 \end{numitem1}
(This time bound is close (even a bit smaller when $|W|> |F|$) to the bound $O(|W|^3|F|^3)$ on the number of oracle calls obtained in~\cite{FZ} to construct the poset of rotation for SBM.)

An efficient method of constructing rotations and their poset in the S-model enables us to efficiently solve the problem of minimizing a linear function over the set of stable matchings, by using an approach similar to the one described in Sect.~\SEC{affine}.

In conclusion, note that when dealing with the S-model for a graph $G$ with possible multiple edges, we can act as described in Remark~1 from Sect.~\SEC{def}, yielding a reduction to a graph $G'$ without multiple edges (equipped with the induced CFs for the old vertices and linear orders for the added ones), and then establish relations between rotations in  the graphs $G$, $G'$ and $\tilde G'$, by arguing as above and using Remark~2 from~\SEC{build} (we omit details here).
\medskip 

\noindent\textbf{Acknowledgment.} The author is thankful to Vladimir Ivanovich Danilov for numerous useful discussions on the topic of this paper and wider, and pointing out to me his paper~\cite{dan}.


\end{document}